\documentclass[10pt]{amsart}
\usepackage{latexsym, amsmath,amssymb, bbold}

	\usepackage{xcolor}

\renewcommand{\epsilon}{\varepsilon}
\newcommand{\newsection}[1]
{\subsection{#1}\setcounter{theorem}{0} \setcounter{equation}{0}
\par\noindent}

\newtheorem{theorem}{Theorem}

\newtheorem{lemma}[theorem]{Lemma}
\newtheorem{corr}[theorem]{Corollary}

\newtheorem{proposition}[theorem]{Proposition}
\newtheorem{deff}[theorem]{Definition}

\newcommand{\bth}{\begin{theorem}}
\newcommand{\ble}{\begin{lemma}}
\newcommand{\bcor}{\begin{corr}}

\newcommand{\bdeff}{\begin{deff}}

\newcommand{\bprop}{\begin{proposition}}
\newcommand{\ele}{\end{lemma}}
\newcommand{\ecor}{\end{corr}}
\newcommand{\edeff}{\end{deff}}

\newcommand{\eprop}{\end{proposition}}

\newcommand{\Rn}{{\mathbb R}^n}

\newcommand{\la}{\lambda}

\newcommand{\e}{\varepsilon}

\renewcommand{\Pi}{\varPi}

\renewcommand{\epsilon}{\varepsilon}

\newcommand{\R}{{\mathbb R}}

\newcommand{\dx}{d_X}
\newcommand{\dy}{d_Y}
\newcommand{\1}{\mathbb{1}}
\newcommand{\laeps}{\mathbb{1}_{[\la-\e(\la), \la+\e(\la)]}}
\newcommand{\el}{\e(\la)}
\newcommand{\exmu}{e^X_{\mu_i}}
\newcommand{\eynu}{e^Y_{\nu_j}}
\newcommand{\munu}{(\mu_i, \nu_j)}
\newcommand{\Ale}{A_{\la,\e}}
\newcommand{\eprod}{e^X_{\mu_i}e^Y_{\nu_j}}
\newcommand{\Ohigh}{\Omega_{high}}
\newcommand{\Olow}{\Omega_{low}}
\newcommand{\Ol}{\Omega_{\ell}}
\newcommand{\sprod}{S^{d_1}\times S^{d_2}}
\newcommand{\ola}{{\mathbb{1}}_{[\la-\la^{-1}, \la+\la^{-1}]}}

\newcommand{\lala}{\mathbb{1}_{[\la-\la^{-1}, \la+\la^{-1}]}}

\newcommand{\mui}{\mu_i}

\begin{document}

\title[Improved Estimates on Product Manifolds]
{Product Manifolds with Improved Spectral Cluster and Weyl Remainder Estimates}
%
%
%
%

\keywords{Eigenfunctions, Weyl formula, spectrum}
\subjclass[2010]{58J50, 35P15}

\thanks{The second author was supported in part by the NSF (DMS-1665373). }

\author[]{Xiaoqi Huang}
\address[X.H.]{Department of Mathematics, University of Maryland, College Park. MD 20742}
\email{xhuang49@umd.edu}

\author[]{Christopher D. Sogge}
\address[C.D.S.]{Department of Mathematics,  Johns Hopkins University,
Baltimore, MD 21218}
\email{sogge@jhu.edu}

\author[]{Michael E. Taylor}
\address[M.E.T.]{Department of Mathematics, University of North Carolina, Chapel Hill, NC 27599}
\email{met@math.unc.edu}

\begin{abstract}
We show that if $Y$ is a compact Riemannian manifold with improved $L^q$ eigenfunction estimates
then, at least for large enough exponents, one always obtains improved $L^q$ bounds on the product
manifold $X\times Y$ if $X$ is another compact manifold.
Similarly, improved Weyl remainder term bounds on the spectral counting function of $Y$
lead to corresponding improvements on $X\times Y$.
The latter results partly generalize  recent ones of Iosevich and Wyman~\cite{IoWy} involving products of spheres.
Also, if $Y$ is a product of five or more spheres, we are able to obtain optimal $L^q(Y)$ and  $L^q(X\times Y)$ eigenfunction
and spectral cluster
estimates for large $q$, which partly addresses a conjecture from \cite{IoWy}
and is related to (and is partly based on) classical bounds  for the number of integer lattice
point on $\la \cdot S^{n-1}$ for $n\ge5$.
\end{abstract}

\maketitle

\centerline{ \bf In memoriam: {\em Robert Strichartz (1943-2021)}}

\newsection{Introduction}


Spectral cluster estimates are operator norm estimates from $L^2$ to $L^q$ of spectral projectors
for the Laplace operator on a compact Riemannian manifold.  In more detail, if $X$ is a compact
Riemannian manifold of dimension $d_X$, and Laplace-Beltrami operator $\Delta_X$, the
universal estimate of \cite{sogge88} has the form
\begin{equation}\label{3}
\bigl\|\, \1_{[\la-1,\la+1]}(P_X)\, \bigr\|_{L^2(X)\to L^q(X)} =O(\la^{\alpha(q)}), \quad P_X =\sqrt{-\Delta_X},
\end{equation}
where, with $n=d_X$,
\begin{equation} \label{4}
\aligned
\alpha(q)=\alpha(q, n) &=\max \bigl(n(\tfrac12 - \tfrac1q)-\tfrac12, \, \tfrac{n -1}2(\tfrac12-\tfrac1q)\bigr)
\\
&=
\begin{cases} n(\tfrac12 - \tfrac1q)-\tfrac12, \quad \text{if } \, \, q\ge q_c(n)=\tfrac{2(n+1)}{n-1},
\\ \\
 \tfrac{n -1}2(\tfrac12-\tfrac1q), \quad \text{if } \, \, 2\le q\le q_c(n).
 \end{cases}
\endaligned
\end{equation}

Much work has been done on the study of special classes of compact Riemannian manifolds for
which stronger estimates hold.  One of our goals here is to show that if $Y$ is a compact
Riemannian manifold, of dimension $d_Y$, with Laplace-Beltrami operator $\Delta_Y$, for which
such stronger results (cf.~\eqref{1}) hold, and if $X$ is an arbitrary compact Riemannian
manifold, as described above,
then, under broad circumstances, the product
manifold $X\times Y$, of dimension $d=d_X+d_Y$,
with product metric tensor and Laplace operator $\Delta=\Delta_X+\Delta_Y$,
also has improved spectral cluster estimates.


We formulate the improved spectral cluster estimates on $Y$ as follows:
\begin{equation}\label{1}
\bigl\| \, \laeps(P_Y) \, \bigr\|_{L^2(Y)\to L^q(Y)} \lesssim \sqrt{\el} \,  B(\la),
\end{equation}
where $P_Y=\sqrt{-\Delta_Y}$, $B(\la)=B(\la,q,Y)$, and $2<q\le \infty$.  Typically $B(\la)$ is $\la$  raised to a power
(see e.g.  \cite{Zygmundlast}, \cite{BourgainDemeterDecouple}, \cite{Hickman}, and \cite{IoWy})
and possibly also involving $\log\la$-powers (see e.g. \cite{Berard}, \cite{SBLog}, \cite{CGBeams}, \cite{HassellTacy}, \cite{SoggeHangzhou}).  So it is natural to assume that
\begin{equation}\label{a}
B(\theta \la)\le C_0 \, B(\la) \quad \text{if } \, \, \la^{-1}\le  \theta \le 2.
\end{equation}
As we shall see later in \eqref{bla}, the improved bound in \eqref{1} requires
\begin{equation}\label{aa}\tag{1.4$'$}
B(\la) \gtrsim \la^{\alpha(q)},\,\,\, \forall \, \, \la \ge 1,
\end{equation}
with $\alpha(q)=\alpha(q,d_Y)$ being the exponent in \eqref{4} depending on the dimension $d_Y$ and $q$.

If $B(\la)=\la^{\alpha(q)}$, with $\alpha(q)$ being the exponent in the universal bounds that are due to the second author \cite{sogge86}, \cite{sogge88} (see \eqref{4}), then
\eqref{1} would represent a $\sqrt{\e(\la)}$ improvement over these bounds by projecting onto bands of size $\approx \e(\la)$
(as opposed to unit sized bands).  We shall always assume that $\e(\la)$ decreases to $0$ as $\la \to \infty$.

We prefer to express improvements including the factor $\sqrt{\e(\la)}$ since this would
match with the bounds for the Stein-Tomas extension operators for $q\ge q_c$ (see \S 4).  If $B(\la)=\la^{\alpha(q)}$ as in \eqref{4} and if
$\e(\la)$ is very small, then they are only possible for certain $q$ larger than the critical
exponent $q_c=\tfrac{2(\dy+1)}{\dy-1}$.
For instance, if $\e(\la)=\la^{-1}$ then one must have
$q\ge \tfrac{2\dy}{\dy-2}$, since otherwise $\sqrt{\e(\la)}\,
\la^{\alpha(q)}\to 0$ as $\la\to \infty$.

We shall also assume that $\e(\la)$ does not go to zero faster than the wavelength of eigenfunctions of frequency $\la$.
More precisely, we shall assume that
\begin{equation}\label{2}
t\to t\e(t)\quad \text{is non-decreasing for} \, \, t\ge 1.
\end{equation}
We note that this is equivalent to the condition that
\begin{equation}\label{2'}\tag{1.5$'$}
\theta \, \e(\theta \la)\le \e(\la) \quad \text{if } \, \, \la \ge 1 \, \, \, \text{and } \, \,
\la^{-1}\lesssim \theta \le 1.
\end{equation}
Typically $\e(\la)=\la^{-\delta}$ for some $\delta\in (0,1]$ or $\e(\la)=(\log\la)^{-\delta}$ for some $\delta>0$.


Here is our first main result.

\begin{theorem}\label{mainthm}  Assume that \eqref{1} is valid with $B(\la)$ and $\e(\la)$ satisfying \eqref{a}--\eqref{aa} and \eqref{2},
respectively.  Then if $\Delta = \Delta_X+\Delta_Y$ is the Laplace-Beltrami operator on the product manifold
and $P=\sqrt{-\Delta}$ we have for all $q>2$
\begin{equation}\label{m.1}
\bigl\| \, \laeps(P)\, \bigr\|_{L^2(X\times Y)\to L^q(X\times Y)}
\lesssim \sqrt{\e(\lambda)} B(\la) \la^{\alpha(q,d_X)+1/2},\ \
\forall \, \la \ge1.
\end{equation}
\end{theorem}

As we shall see at the end of the next section, if  \eqref{1} is an improvement on $Y$ over the universal bounds in \cite{sogge88}, then, at
least for sufficiently large exponents, \eqref{m.1} says that there are improved $L^q$ spectral projection estimates on $X\times Y$.


Our second main result deals with the Weyl law, for the specral counting function of the Laplace operator.
Recall that the universal Weyl formula of Avakumovic~\cite{Avakumovic},
Levitan~\cite{Levitan} and H\"ormander~\cite{HSpec} states that if $N(X,\lambda)$ denotes the number of
eigenvalues, counted with multiplicity, of $P_X$ which are $\le \lambda$ then
\begin{equation}\label{3.1}
N(X,\lambda)=(2\pi)^{-d_X} \omega_{d_X}  (\text{Vol}\,X) \, \lambda^{d_X} +O(\lambda^{d_X-1}),
\end{equation}
with $\omega_{n}$ denoting the volume of the unit ball in $\Rn$.

This result cannot be improved if $X$ is the
round sphere of dimension $d_X$.  On the other hand, there are a number of results that do yield
improved Weyl remainder estimates.  In \cite{DuistermaatGuillemin} it is shown that one can improve
$O(\lambda^{d_X-1})$ to $o(\lambda^{d_X-1})$ in case the set of periodic geodesics has measure zero.
The paper \cite{Berard} shows that under certain geometrical conditions, such as nonpositive curvature,
one can improve the remainder estimate to $O(\lambda^{d_X-1}/\log \lambda)$.  Recently, Canzani and Galkowski
\cite{CZWeyl} obtained such an improved remainder estimate for a much broader class of Riemannian manifolds.
Among the results obtained there is that one gets this $(\log \lambda)^{-1}$ improvement on each Cartesian
product manifold, with the product metric.  Iosevich and Wyman \cite{IoWy} showed there are power improvements
for products of spheres.

As with the $L^q$ improvements in Theorem~\ref{mainthm}, we shall assume that there are $\e(\la)$ improvements, with $\e(\la)$ as in
\eqref{2}--\eqref{2'}, for $Y$,  and show that these carry over for $X\times Y$.  So, we shall assume that
\begin{equation}\label{3.2}
N(Y,\lambda)=(2\pi)^{-d_Y} \omega_{d_Y}  (\text{Vol}\,Y) \,  \lambda^{d_Y} +R_Y(\la),  \quad \text{with } \, \, \,  R_Y(\la)=O(\e(\la) \, \lambda^{d_Y-1}) .
\end{equation}

Here is our second main result.

\begin{theorem}\label{weylthm}
Assume that \eqref{3.2} is valid.  Then, for each compact Riemannian manifold, of dimension $d_X$,
\begin{equation}\label{3.3}
N(X\times Y,\lambda)=(2\pi)^{-d} \omega_{d}\,  \text{\rm Vol}(X\times Y)\,  \lambda^{d} +O(\e(\la) \, \lambda^{d-1}), \quad d=d_X+d_Y.
\end{equation}
\end{theorem}

The structure of the rest of this paper is the following.  In \S{2} we prove Theorem \ref{mainthm}
and give applications.  In \S{3} we prove Theorem \ref{weylthm}.  Section 4 presents some further results,
including a study of multiple products of spheres.

\newsection{Proof of $L^q$--improvements and some applications}

Let us now turn to the proof of Theorem~\ref{mainthm}.  We first choose an orthonormal basis, $\{\exmu\}$,  of
eigenfunctions of $P_X$ with eigenvalues $\mu_i$ and $\{\eynu\}$ of $P_Y$ with eigenvalues $\nu_j$.  Thus,
\begin{equation}\label{2.9}
-\Delta_X \exmu = \mu^2_i \exmu \, \, \, \text{and } \, \, \,
-\Delta_Y \eynu =\nu^2_j \eynu.
\end{equation}
Then $\{ \exmu(x)\cdot \eynu(y)\}_{i,j}$ is an orthonormal basis of eigenfunctions of $P=\sqrt{-\Delta}$, where $\Delta=\Delta_X+\Delta_Y$,
with eigenvalues $\sqrt{\mu_i^2+\nu_j^2}$, i.e.,
\begin{equation}\label{2.10}
-\Delta (\exmu \eynu)=(\mu_i^2+\nu_j^2) \exmu\eynu.
\end{equation}

Consequently, the first inequality, \eqref{m.1}, is equivalent to the following
\begin{equation}\label{2.11}
\Bigl\| \, \sum_{\munu \in \Ale} a_{ij} \, \eprod \, \Bigr\|_{L^q(X\times Y)}\lesssim B(\la)
\la^{\alpha(q)}\, \sqrt{\la \e(\la)} \, \|a\|_{\ell^2}
\end{equation}
with $\|a\|_{\ell^2} =(\sum_{i,j}|a_{ij}|^2)^{1/2}$ and $\Ale$ denoting the $\el$-annulus about $\la\cdot S^2$, i.e.,
\begin{equation}\label{2.12}
\Ale =\{ \, (\mu,\nu): \, |\la-\sqrt{\mu^2+\nu^2}|\le \e(\la)\}.
\end{equation}

To be able to use our assumptions \eqref{1} and \eqref{2'} and the universal bounds \eqref{3} it is natural
to break up the annulus into several pieces.  Specifically,
let
\begin{equation}\label{2.13}
\Ohigh = \{ (\mu,\nu)\in \Ale: \, |\nu|\ge \la/2\},
\end{equation}
denote the portion of $\Ale$ where $|\nu|$ is relatively large and
\begin{equation}\label{2.14}
\Olow = \{ (\mu,\nu)\in \Ale: \, |\nu|\le 1\},
\end{equation}
be the portion where it is relatively small. We shall also
break up the remaining region
into the following disjoint dyadic pieces
\begin{equation}\label{2.15}
\Ol = \{(\mu,\nu) \in \Ale \backslash \Ohigh : \, \, |\nu|\in (\la 2^{-\ell}, \, \la 2^{-\ell +1}]\}.
\end{equation}
Thus,
\begin{equation}\label{2.16}
\Ale = \Ohigh \cup   \Olow \cup\,  \Bigl( \bigcup_{2^{-\ell} \in [\la^{-1}, \frac14]}
\Ol \, \Bigr).
\end{equation}

We shall use the following simple lemma describing the geometry of each of these pieces.

\begin{lemma}\label{lemma1}  There is a uniform constant $C_0$ so that
\begin{equation}\label{2.17}
|\sqrt{\la^2-\mu^2}-\nu|\le C_0 \e(\la) \quad
\text{if } \, \, (\mu,\nu)\in \Ohigh \, \, \text{and } \, \, \mu,\nu\ge0.
\end{equation}
Also if $\Ol\ne \emptyset$ then
\begin{equation}\label{2.18}
|\sqrt{\la^2-\mu^2}-\nu|\le C_0 2^\ell \e(\la) \quad
\text{if } \, \, (\mu,\nu)\in \Ol \, \, \text{and } \, \, \mu,\nu\ge0,
\end{equation}
and if
\begin{equation}\label{il}
I_\ell =\{\mu: \, (\mu,\nu)\in \Ol, \, \mu,\nu\ge 0\},
\end{equation}
then for fixed $\ell$ with $2^{-\ell}\ge \la^{-\frac12}$,
\begin{equation}\label{2.19}
I_\ell \quad \text{is an interval in } \, \, [0,\la] \quad \text{of length } \, \,
|I_\ell|\le C_0\la 2^{-2\ell},
\end{equation}
and also
\begin{multline}\label{2.20}
|\sqrt{\la^2-\nu^2}-\mu|\le C_0 \quad \text{if } \, \, \,  \mu,\nu\ge 0, \\
\text{and } \, \, (\mu,\nu)\in \Ol \,\,\,\text{with}\,\,2^{-\ell}\le \la^{-\frac12}, \,\,\,\text{or}\,\,\, (\mu,\nu)\in \Olow .
\end{multline}
\end{lemma}

The bound in \eqref{2.17} is straightforward since $\nu\ge \la/2$ for the $(\mu,\nu)$ there.
One obtains \eqref{2.20} similarly and indeed can replace $C_0$ there by $C_0 \e(\la)$,
although this will not be needed.  One obtains \eqref{2.18} and \eqref{2.19} by noting that
the $(\mu,\nu)$ there must be of the form
$$(\mu,\nu)=r(\cos\theta,\sin\theta) \quad
\text{with } \, \, r\in [\la-\e(\la),\la+\e(\la)] \quad \text{and } \, \theta\approx 2^{-\ell}.$$

We also require the following estimates which are a simple consequence of our main assumption \eqref{1} and
the universal bounds \eqref{3}.

\begin{lemma}\label{lemma2}
There is a universal constant $C_0$ so that
\begin{equation}\label{2.21}
\bigl\| \, \1_{[\la-\rho,\la+\rho]}(P_X)\, \bigr\|_{L^2(X)\to L^q(X)} \le C_0
\rho^{1/2} \, \la^{\alpha(q)}, \, \, \, \text{if } \, \, \rho\in [1,\la],
\end{equation}
and
\begin{equation}\label{2.22}
\bigl\| \, \1_{[\la-\rho,\la+\rho]}(P_Y)\, \bigr\|_{L^2(Y)\to L^q(Y)} \le C_0
\rho^{1/2} \, B(\la), \, \, \, \text{if } \, \, \rho\in [\e(\la),\la],
\end{equation}
and also
\begin{equation}\label{2.23}
\big\| \, \1_{[2^k, 2^{k+1}]}(P_Y)\, \bigl\|_{L^2(Y)\to L^q(Y)}
\le C_0 2^{\dy(\frac12-\frac1q)k}.
\end{equation}
\end{lemma}
Note that if we let $\rho=1$ in \eqref{2.22}, we have
\begin{equation}
\begin{aligned}\label{bla}
 B(\la)&\gtrsim \bigl\| \, \1_{[\la-1,\la+1]}(P_Y)\, \bigr\|_{L^2(Y)\to L^q(Y)} \\
 & \gtrsim \la^{\alpha(q)},
\end{aligned}
 \end{equation}
 where in the second inequality we used the lower bounds on the spectral projection operator, which holds in the general case(see  \cite{soggefica}).
\begin{proof}
The proofs are well known.
One
 obtains \eqref{2.21} from \eqref{3} by writing $[\la-\rho,\la+\rho]$ as the union of
$O(\rho)$ disjoint intervals $I_k$ of length $1$ or less, each contained in $[0,2\la]$.  By the Cauchy-Schwarz inequality one then has
\begin{align*}\bigl\| \, \1_{[\la-\rho,\la+\rho]}(P_X)f\, \bigr\|_{L^q(X)}
&\lesssim \rho^{1/2} \bigl(\sum_k \bigl\| \, \1_{I_k}(P_X)f\, \bigr\|_{L^q(X)}^2 \, \bigr)^{1/2}
\\
&\lesssim   \rho^{1/2} \, \la^{\alpha(q)}  \bigl(\sum_k \bigl\| \, \1_{I_k}(P_X)f\, \bigr\|_{L^2(X)}^2 \, \bigr)^{1/2}
\\
&\lesssim  \rho^{1/2}\, \la^{\alpha(q)}\|f\|_{L^2(X)},
\end{align*}
using \eqref{3} in the second inequality and
orthogonality in the last one.

To prove \eqref{2.22} we note that our assumptions \eqref{a} and \eqref{2'} imply that $B(\la_1)\approx B(\la_2)$ and $\e(\la_1)\approx \e(\la_2)$ if $\la_1\approx \la_2$.
Taking this into account, if $\rho \in [\e(\la),1]$ one proves \eqref{2.22} by a similar argument used for \eqref{2.21}  if one covers $[\la-\rho,\la+\rho]$ by
$O(\rho/\e(\la))$ disjoint intervals of length $\e(\la)$ or less and uses the fact that \eqref{1} includes a $(\e(\la))^{1/2}$ factor in
the right.  Similarly, if $I_k=[k,k-1)$ with $k\le \la$ then one concludes that
$$\|\1_{I_k}\|_{L^2(Y)\to L^q(Y)} \lesssim B(k) \lesssim B(\la),$$
which can be used with the above argument involving the use of the Cauchy-Schwarz inequality to handle the case where $\rho\in [1,\la]$.

The remaining inequality, \eqref{2.23}, is a standard Bernstein estimate (see \cite{soggefica}).
\end{proof}

Having collected the tools we need, let us state the bounds associated with the various regions \eqref{2.13}--\eqref{2.15} of
$\Ale$ that will give us \eqref{2.11}.

First for all $\e(\la)$ satisfying the conditions in Theorem~\ref{mainthm} we claim that we have
\begin{equation}\label{2.24}
\Bigl\| \, \sum_{\munu \in \Ohigh} a_{ij} \, \eprod \, \Bigr\|_{L^q(X\times Y)}\lesssim  \sqrt{\e(\la)} \, B(\la) \cdot
\la^{\alpha(q)}\, \sqrt{\la} \, \|a\|_{\ell^2},
\end{equation}
\begin{equation}\label{2.24'}
\Bigl\| \, \sum_{\munu \in \Olow} a_{ij} \, \eprod \, \Bigr\|_{L^q(X\times Y)}\lesssim
\la^{\alpha(q)}\, \, \|a\|_{\ell^2},
\end{equation}
as well as
\begin{equation}\label{2.25}
\Bigl\| \, \sum_{\munu \in \Ol} a_{ij} \, \eprod \, \Bigr\|_{L^q(X\times Y)}\lesssim \sqrt{2^\ell  \e(\la)}  \, B(2^{-\ell}\la) \cdot
\la^{\alpha(q)} \, \sqrt{\la 2^{-2\ell}}
\, \, \|a\|_{\ell^2},
\end{equation}
if $ 2^{-\ell}\ge \la^{-\frac12}$.

For remaining pieces $ 2^{-\ell}\le \la^{-\frac12}$, we need to obtain estimates to handle  the two cases where $\e(\la)\le \frac14\cdot\la 2^{-2\ell}$ and $\e(\la)\ge \frac14\cdot\la 2^{-2\ell}$.
For the first case, we shall use the following estimate which is valid for all $\e(\la)$ satisfying \eqref{2'}
\begin{equation}\label{2.26}
\Bigl\| \, \sum_{\munu \in \Ol} a_{ij} \, \eprod \, \Bigr\|_{L^q(X\times Y)}
\lesssim
 \sqrt{2^\ell  \e(\la)}  \, B(2^{-\ell}\la) \cdot
\la^{\alpha(q)} \,
  \|a\|_{\ell^2},
\end{equation}
while for the other case we shall use the fact that we also always have
\begin{equation}\label{2.27}
\Bigl\| \, \sum_{\munu \in \Ol} a_{ij} \, \eprod \, \Bigr\|_{L^q(X\times Y)}
\lesssim \la^{\alpha(q)} \, (\la2^{-\ell})^{\dy(\frac12-\frac1q)}
\|a\|_{\ell^2}.
\end{equation}

Let us now see how the bounds in \eqref{2.24}--\eqref{2.27} yield those in Theorem~\ref{mainthm}.  To use \eqref{2.25} we note
that by \eqref{a} with $\theta =2^{-\ell}$
$$\sqrt{2^\ell \e(\la)} \,  B(2^{-\ell}\la) \, \la^{\alpha(q)} \, \sqrt{\la 2^{-2\ell}}
\lesssim 2^{-\ell/2} \, B(\la) \, \la^{\alpha(q)} \, \sqrt{\la \e(\la)},
$$
and so
\begin{equation}\label{2.28}
\sum_{2^{-\ell}\in [\la^{-\frac12}, \frac14]} \Bigl\|
\, \sum_{\munu \in \Ol  } a_{ij}\exmu\eynu \, \Bigr\|_{L^q(X\times Y)} \lesssim B(\la) \, \la^{\alpha(q)}\,
\sqrt{\la \e(\la)} \, \|a\|_{\ell^2}.
\end{equation}

To use \eqref{2.26} we note
that by \eqref{a} with $\theta =2^{-\ell}$,
$$\sqrt{2^\ell \e(\la)} \,  B(2^{-\ell}\la) \, \la^{\alpha(q)}
\lesssim 2^{\ell/2} \, B(\la) \, \la^{\alpha(q)} \, \sqrt{ \e(\la)},
$$
and so
\begin{equation}\label{2.28'}
\sum_{2^{-\ell}\in [2\la^{-\frac12}\e(\la)^{\frac12}, \la^{-\frac12}]} \Bigl\|
\, \sum_{\munu \in \Ol  } a_{ij}\exmu\eynu \, \Bigr\|_{L^q(X\times Y)} \lesssim B(\la) \, \la^{\alpha(q)}\,
(\la\e(\la))^{1/4} \, \|a\|_{\ell^2},
\end{equation}
which is better than desired since we are assuming $\e(\la)\ge \la^{-1}$.

Finally, if $2^{-\ell}\in [\la^{-1}, 2\la^{-\frac12}\e(\la)^{\frac12}]$, by \eqref{2.27}, we have
\begin{equation}\label{2.28''}
\sum_{2^{-\ell}\in  [\la^{-1}, 2\la^{-\frac12}\e(\la)^{\frac12}]} \Bigl\|
\, \sum_{\munu \in \Ol  } a_{ij}\exmu\eynu \, \Bigr\|_{L^q(X\times Y)} \lesssim  \la^{\alpha(q)} \, (\la\e(\la))^{\frac{\dy}{2}(\frac12-\frac1q)}
\|a\|_{\ell^2}.
\end{equation}
It is straightforward to check that for all $q\ge 2$, $$(\la\e(\la))^{\frac{\dy}{2}(\frac12-\frac1q)}\lesssim  B(\la)\,
\sqrt{\la \e(\la)}, $$ given the fact that $B(\la)\ge \la^{\alpha(q)}$ and $\la^{-1}\le\e(\la)\le 1$. Thus, our proof would
be complete if we could establish \eqref{2.24}--\eqref{2.27}.

To prove the first one, \eqref{2.24}, we note that if $y\in Y$ is fixed, since $\mu_i\le \la$ if $\munu\in \Ohigh$, by
\eqref{2.21} with $\rho=\la$ and orthogonality
\begin{align*}
\Bigl\| \, \sum_{\munu\in\Ohigh} a_{ij}\exmu(\, \cdot\, )\, \eynu(y)\, \Bigr\|_{L^q(X)}
&\lesssim \la^{\alpha(q)+1/2}
\Bigl\| \, \sum_{\munu\in\Ohigh} a_{ij}\exmu(\, \cdot\, )\, \eynu(y)\, \Bigr\|_{L^2(X)}
\\
&=\la^{\alpha(q)+1/2} \,
\Bigl( \, \sum_i \, \Bigl| \, \sum_{\{j: \, \munu\in \Ohigh \}} a_{ij}\eynu(y) \, \Bigr|^2 \, \Bigr)^{1/2}.
\end{align*}
If we take the $L^q(Y)$ norm of the left side and use this inequality along with Minkowski's inequality we conclude that
\begin{equation} \label{2.29}
\aligned
&\Bigl\| \, \sum_{\munu\in\Ohigh} a_{ij}\exmu\eynu \Bigr\|_{L^q(X\times Y)}
\\
&\lesssim \la^{\alpha(q)+1/2}
\Bigl( \, \sum_i \, \Bigl\| \, \sum_{ \{j: \, \munu\in \Ohigh\} } a_{ij}\eynu \, \Bigr\|_{L^q(Y)}^2 \, \Bigr)^{1/2}.
\endaligned
\end{equation}
Since $\nu_j\approx \la$ if $\munu\in \Ohigh$, by \eqref{2.17} and \eqref{2.22} with $\rho=C_0\e(\la)$, we have for each fixed $i$
\begin{align}\label{2.30}
\Bigl\| \, \sum_{ \{j: \, \munu\in \Ohigh\} } a_{ij}\eynu \, \Bigr\|_{L^q(Y)} &\lesssim \sqrt{\e(\la)} B(\la) \,
\Bigl\| \, \sum_{ \{j: \, \munu\in \Ohigh\} } a_{ij}\eynu \, \Bigr\|_{L^2(Y)}
\\
&=\sqrt{\e(\la)} B(\la) \, \Bigl( \, \sum_{\{j: \, \munu\in \Ohigh \} } |a_{ij}|^2\, \Bigr)^{1/2}. \notag
\end{align}
Clearly \eqref{2.29} and \eqref{2.30} imply \eqref{2.24}.

The proof of \eqref{2.25} is similar.  Recall that the nonzero terms involve $\mu_i\in I_\ell$, if $2^{-\ell}\ge\la^{-\frac12}$, then $I_\ell$ is an interval
of length $\rho\le C_0\la 2^{-2\ell}$ as in \eqref{il}.  So, if we use the analog of \eqref{2.21} with this value of $\rho$ and  with
$\la$ replaced by the center of $I_\ell$, we can repeat the proof of \eqref{2.29} to conclude that
\begin{equation} \label{2.31}
\aligned
&\Bigl\| \, \sum_{\munu\in\Ol} a_{ij}\exmu\eynu \Bigr\|_{L^q(X\times Y)}
\\
&\lesssim \la^{\alpha(q)} \, \sqrt{\la 2^{-2\ell}}\,
\Bigl( \, \sum_i \, \Bigl\| \, \sum_{ \{j: \, \munu\in \Ol\} } a_{ij}\eynu \, \Bigr\|_{L^q(Y)}^2 \, \Bigr)^{1/2}.
\endaligned
\end{equation}
Since $\nu_j\approx 2^{-\ell}\la$ if $\munu\in \Ol$, and by \eqref{2'} together with the fact that $2^{-\ell}\ge\la^{-\frac12}$, we have
\begin{equation}\label{condition}
\e(2^{-\ell}\la)\le2^\ell \e(\la)\le 2^{-\ell}\la.
\end{equation}
By \eqref{2.18} and \eqref{2.22} with $\rho=2^\ell \e(\la)$, we can argue
as above to see that for each fixed $i$ we have
\begin{equation}\label{2.32}
\Bigl\| \, \sum_{ \{j: \, \munu\in \Ol\} } a_{ij}\eynu \, \Bigr\|_{L^q(Y)}
\\
\lesssim \big(2^\ell \e(\la)\bigr)^{1/2} \, B(2^{-\ell} \la)\,
\Bigl( \, \sum_{ \{j: \, \munu\in \Ol\} } |a_{ij}|^2 \, \Bigr)^{1/2}.
\end{equation}
By combining \eqref{2.31} and \eqref{2.32} we obtain \eqref{2.25}.

Next, we turn to \eqref{2.26}.  We note that if $2^{-\ell}\le\la^{-\frac12}$, there is a uniform constant $C_0$ so that
\begin{equation}\label{2.33}
\mu_i \in [\la-C_0,\la+C_0], \quad \text{if } \, \, \munu\in \Ol \, \, \, \text{for some }  \, j.
\end{equation}
This just follows from the fact that if $\munu\in \Olow$ then we can write
$\munu=r(\cos\theta,\sin\theta)$ with $0\le \theta \lesssim \la^{-1/2}$ and $r\in [\la-\e(\la),\la+\e(\la)]\subset [\la-1,\la+1]$.
If we use \eqref{2.33} and \eqref{3} we can argue as above to see that
\begin{equation}\label{2.34}
\Bigl\| \, \sum_{\munu\in\Ol} a_{ij}\exmu\eynu \Bigr\|_{L^q(X\times Y)}
\lesssim \la^{\alpha(q)}
\Bigl( \, \sum_i \, \Bigl\| \, \sum_{ \{j: \, \munu\in \Ol\} } a_{ij}\eynu \, \Bigr\|_{L^q(Y)}^2 \, \Bigr)^{1/2}.
\end{equation}
For fixed $\mu_i$,  if $ \munu\in \Ol$, we have $\nu_j\approx 2^{-\ell}\la$, and by  \eqref{2.18},  $\nu_j$ lie in a interval of length $\e(\la)2^\ell$.
Now using \eqref{2'} again together with the fact that $\e(\la)\le \frac14 \la 2^{-2\ell}$, one can see that \eqref{condition} still hold in this case. Thus by  \eqref{2.22} with $\rho=2^\ell \e(\la)$, we can argue
as above to see that for each fixed $i$, we have the analogous inequality as in
\eqref{2.32}, which, combined with \eqref{2.34},  implies \eqref{2.26}.

To prove \eqref{2.27}, if one uses
\eqref{2.23} then we find we can replace \eqref{2.32} with
$$
\Bigl\| \, \sum_{ \{j: \, \munu\in \Ol\} } a_{ij}\eynu \, \Bigr\|_{L^q(Y)}
\lesssim (\la2^{-\ell})^{\dy(\frac12-\frac1q)}   \,
\Bigl( \, \sum_{ \{j: \, \munu\in \Ol\} } |a_{ij}|^2 \, \Bigr)^{1/2},
$$
this along with \eqref{2.34} yields \eqref{2.27}.

The proof of \eqref{2.24'} is similar. Since in this case, there is a uniform constant $C_0$ so that
\begin{equation}\label{2.33'}
\mu_i \in [\la-C_0,\la+C_0], \quad \text{if } \, \, \munu\in \Olow \, \, \, \text{for some }  \, j.
\end{equation}
Thus \eqref{2.34} still holds in this case, and
 by \eqref{2.23}, we can replace
 \eqref{2.32} with
\begin{equation}\label{2.35}
\Bigl\| \, \sum_{ \{j: \, \munu\in \Olow\} } a_{ij}\eynu \, \Bigr\|_{L^q(Y)}
\lesssim
\Bigl( \, \sum_{ \{j: \, \munu\in \Olow\} } |a_{ij}|^2 \, \Bigr)^{1/2}.
\end{equation}
this along with \eqref{2.34} yields \eqref{2.24'}.

\subsubsection{{Some applications}}

Let us now show that for products of round spheres $S^{d_1}\times S^{d_2}$ one can obtain power improvements
over the universal bounds in \cite{sogge88} for {\em all} exponents $2<q\le \infty$.  This generalizes the
$L^\infty$ improvements of Iosevich and Wyman~\cite{IoWy}.
Using our improved $L^q$-estimates we can also obtain
improved bounds for large exponents for products of the form $S^{d_1}\times S^{d_2}\times M^n$ where $M^n$
is an arbitrary compact manifold of dimension $n$.  If $M^n$ is a product of spheres and $q=\infty$ the bounds agree
with the ones that are implicit in Iosevich and Wyman~\cite{IoWy}.

\begin{theorem}\label{spthm}  Suppose that $d_1, d_2\ge1$.  Then for all $\e>0$ we have the following estimates for
eigenfunctions on $S^{d_1}\times S^{d_2}$
\begin{equation}\label{a.1}
\|e_\la\|_{L^q(S^{d_1}\times S^{d_2})} \le C_\e \,
\la^{\alpha(q,d_1)+\alpha(q,d_2)+\e} \, \|e_\la\|_{L^2(S^{d_1}\times S^{d_2})}, \quad
2<q\le \infty,
\end{equation}
where
\begin{equation}\label{a.2}
\alpha(q,d)=\max\bigl( \, d(\tfrac12-\tfrac1q)-\tfrac12, \, \tfrac{d-1}2(\tfrac12-\tfrac1q)\, \bigr)
\end{equation}
is the $\lambda$-exponent in the $d$-dimensional universal bounds.
\end{theorem}

In order to use Theorem~\ref{mainthm} to obtain bounds for product manifolds involving $\sprod$ we note that the distinct
eigenvalues of $P=\sqrt{-\Delta_{\sprod}}$ are of the form
$$\sqrt{(k+(d_1-1)/2)^2 +(\ell +(d_2-1)/2)^2} \quad \text{ with  } \, \, k,\ell =0,1,2,\dots .$$
Consequently the gap between successive distinct eigenvalues which are comparable to $\la$ must be larger
than a fixed multiple of $\la^{-1}$.
 So, \eqref{a.1} implies that we also have the corresponding bounds for the spectral projection
operators onto windows of length $\e(\la)=\la^{-1}$:
\begin{equation}
\label{a.1'}\tag{2.35$'$}
\aligned
&\bigl\|  \ola(P)\, \bigl\|_{L^2(\sprod) \to L^q(\sprod)} \le C_\e
\la^{\alpha(q,d_1)+\alpha(q,d_2)+\e}, \, \, \forall \, \e>0,
\\
&\text{if } \, \, 2<q\le \infty \, \, \, \text{and } \, \, P=\sqrt{-\Delta_{\sprod}}.
\endaligned
\end{equation}

A calculation shows that for {\em all} $q>2$
$$\alpha(q,d_1)+\alpha(q,d_2)<\alpha(q,d_1+d_2).$$
  Thus \eqref{a.1'} says that on $\sprod$ one has power improvements over the universal bounds
for manifolds of dimension $d=d_1+d_2$ (but with $\e(\la)\equiv 1$).   Also, by considering tensor products of sphercial
harmonics that saturate the $L^q(S^{d_j})$, $j=1,2$, bounds (see e.g. \cite{sogge86}) one sees that \eqref{a.1} and
hence \eqref{a.1'} are optimal (up to possibly the $\la^\e$ factor).

Let us now single out a couple of special cases of \eqref{a.1'}.

First, we have
\begin{equation}
\label{a.3}
\aligned
&\bigl\| \, \ola(P)\, \bigr\|_{L^2(\sprod)\to L^q(\sprod)} \le C_\e
\la^{d(\frac12-\frac1q)-\frac12} \, \la^{-\frac12+\e}, \, \, \forall \, \, \e>0,
\\
&\text{if } \, \, d=d_1+d_2 \, \, \, \text{and } \, \,
q\ge \max \bigl( \, \tfrac{2(d_1+1)}{d_1-1}, \, \tfrac{2(d_2+1)}{d_2-1}\, \bigr),
\endaligned
\end{equation}
which is a $\la^{-\frac12+\e}$ improvement for this range of exponents in dimension $d$ versus the universal bounds.
This is optimal in the sense that no bounds of this type may hold on {\em any} manifold of dimension
$d$ with $\la^{-\frac12+\e}$ replaced by $\la^{-\frac12-\delta}$ for some $\delta>0$.  For, by Bernstein inequalities, such
an estimate would imply that the above spectral projection operators map $L^2\to L^\infty$ with norm
$O(\la^{\frac{d}2-1-\delta})$.  This cannot hold in dimension $d$ since it would imply that the number of eigenvalues of the square root
of minus the Laplacian
counted with multiplicity which are in subintervals of length $\la^{-1}$
in $[\la/2,\la]$ would be $O(\la^{d-2-\delta})$, and this would contradict the Weyl formula for $P$.

Second, we note that if $d_1=d_2=1$ and $q=\infty$, then \eqref{a.3} is just
\begin{equation}\label{a.4}
\bigl\| \, \ola(\sqrt{-\Delta_{{\mathbb T}^2})} \, \bigr\|_{L^2({\mathbb T}^2) \to L^\infty({\mathbb T}^2)}
=O(\la^\e), \, \, \forall \e>0.
\end{equation}
This is equivalent to the classical fact that the number of integer lattice points on $\la \cdot S^1$ is $O(\la^\e)$, i.e.,
\begin{equation}\label{a.4'} \tag{2.38$'$}
\#  \{j\in {\mathbb Z}^2: \, |j|=\la\} =O(\la^\e) \quad  \forall \, \e>0.
\end{equation}
We shall use this bound in our proof of Theorem~\ref{spthm}.

Before proving this result, let us show how we can use the bounds in Theorem~\ref{mainthm} to obtain a
couple of corollaries.

The first says that sufficiently large exponents we can obtain power improvements of the universal bounds
for products involving $\sprod$.

\begin{corr}\label{cor1}
Let $M^n$ be a compact manifold of dimension $n\ge1$ and consider the product
manifold $\sprod\times M^n$ where $d_1,d_2\ge 1$.  Then if
$$P=\sqrt{-(\Delta_{S^{d_1}}+\Delta_{S^{d_2}}+\Delta_{M^n})},$$
we have
\begin{equation}
\label{a.5}
\aligned
&\bigl\| \, \ola(P)\, \bigr\|_{L^2(\sprod\times M^n)\to L^q(\sprod\times M^n)}
\le C_\e \la^{d(\frac12-\frac1q)-\frac12} \, \la^{-\frac12+\e},
\\
&\forall \e>0, \, \, \, \text{if } \, \, \, d=d_1+d_2+n \, \, \,
\text{and } \, \, q\ge \max\bigl( \, \tfrac{2(d_1+1)}{d_1-1}, \, \tfrac{2(d_2+1)}{d_2-1},
\, \tfrac{2(n+1)}{n-1} \, \bigr).
\endaligned
\end{equation}
Furthermore, we have
\begin{equation}\label{a.6}
\bigl\| \, \ola(P)\, \bigr\|_{L^2(\sprod\times M^n)\to L^q(\sprod\times M^n)}
\le C_q \la^{d(\frac12-\frac1q)-\frac12 -\delta} ,
\end{equation}
for some $\delta=\delta(q,d_1,d_2,n)>0$ if $q>\tfrac{2(d+1)}{d-1}$, $d=d_1+d_2+n$.
\end{corr}

To prove these to bounds we note that for $q$ as in \eqref{a.5} we have
$$\alpha(q,d_1)+\alpha(q,d_2)+\alpha(q,n)+\tfrac12 = d(\tfrac12-\tfrac1q)-1.$$
Consequently, \eqref{a.5} follows immediately from \eqref{a.1} and \eqref{m.1} with
$\e(\la)=\la^{-1}$ and
$$B(\la)=\la^{(d_1+d_2)(\frac12-\frac1q)-\frac12+\e}.$$
Since \eqref{a.5} is a power improvement over $O(\la^{d(\frac12-\frac1q)-\frac12})$ bounds for large exponents
and the universal bounds imply that \eqref{a.6} is valid when $\delta=0$ and
$q=q_c(d)=\tfrac{2(d+1)}{d-1}$, one obtains \eqref{a.6} via a simple interpolation argument.

A calculation show that we cannot use use Theorem~\ref{mainthm} to obtain improvements over the universal bounds
when $q\in (2, q_c(d)]$ with $q_c(d)$ as above being the critical exponent.  We should
point our that Canzani and Galkowski~\cite{CGBeams} recently obtained $\sqrt{\log\la}$-improvements over the universal
bounds (with $\e(\la)=(\log \la)^{-1}$) for arbitrary products of manifolds and $q>q_c(d)$.  They as well as
Iosevich and Wyman~\cite{IoWy} conjectured that for such manifolds appropriate power improvements over
the universal bounds should always be possible.  Obtaining any improvements for
$q\in (2,q_c(d)]$, though, appears difficult except in special cases such as for products involving products
of spheres as above.  Perhaps, though, the Kakeya-Nikodym approach that was used in \cite{BlairSoggeToponogov} and \cite{SBLog} to obtain
log-power improvements of eigenfunction estimates for manifolds of nonpositive sectional curvature  could be used
to handle critical and subcritical exponents.

Let us also state one more corollary which generalizes the well known higher dimensional version of
\eqref{a.4'}:
\begin{equation}\label{a.7}
\#\{j\in {\mathbb Z^n}: \, |j|=\la\} =O(\la^{n-2+\e}) \, \, \forall \, \e>0 \, \,
\text{if } \, \, n\ge3.
\end{equation}
Just as for the special case where $n=2$ discussed above, this is easily seen to be equivalent  to
the following sup-norm bounds
\begin{equation}\label{a.7'}\tag{2.41$'$}
\bigl\| \ola(\sqrt{-\Delta_{{\mathbb T}^n}} ) \, \bigl\|_{L^2({\mathbb T}^n)\to L^q({\mathbb T}^n)}
=O(\la^{\frac{n-1}2} \, \la^{-\frac12+\e})
\, \, \forall \, \e>0 \, \, \text{if } \, n\ge3.
\end{equation}

If we use Theorem~\ref{mainthm} for $q=\infty$ with $\e(\la)=\la^{-1}$ and
$B(\la)=\la^{\frac{n-1}2+\e}$ we can argue as above to obtain the following generalization of \eqref{a.7'}.

\begin{corr}\label{cor2}  Let $M^{n-2}$ be a compact Riemannian manifold of dimension $n-2$ where
$n\ge3$.  Then if
$$P=\sqrt{-(\Delta_{{\mathbb T}^2}+\Delta_{M^{n-2}})}$$
is the square root of minus the Laplacian on the $n$-dimensional product manifold ${\mathbb T}^2\times M^{n-2}$ we have
\begin{equation}\label{a.8}
\bigl\| \ola(P ) \, \bigl\|_{L^2({\mathbb T}^2 \times M^{n-2})\to L^q({\mathbb T}^2 \times M^{n-2})}
=O(\la^{\frac{n-1}2} \, \la^{-\frac12+\e})
\, \, \forall \, \e>0.
\end{equation}
Consequently if $0=\la_0\le \la_1\le \la_2\le \dots$ are the eigenvalues of $P$
\begin{equation}\label{a.9}
\#\{\la_j\in [\la-\la^{-1},\la+\la^{-1}]\} =O(\la^{n-2+\e}), \, \, \forall \, \e>0.
\end{equation}
\end{corr}

The first estimate, \eqref{a.8}, follows from \eqref{a.4} and Theorem~\ref{mainthm}.  As is well known
(see e.g., \cite{SoggeHangzhou}) it implies the counting bounds \eqref{a.9}.  These are optimal, since
as we discussed before, $O(\la^{n-2-\delta})$ with $\delta>0$ bounds cannot hold due to the Weyl formula.

One can also obtain power improvements for products $X\times {\mathbb T}^n$
 using the following ``discrete restriction theorem''   of Bourgain and Demeter \cite{BourgainDemeterDecouple}
 (toral eigenfunction bounds):
 \begin{equation}\label{9}
 \bigl\| \1_{[\la-\la^{-1}, \la+\la^{-1}]}\, (\sqrt{-\Delta_{{\mathbb T}^n}})\,
 \bigr\|_{L^2({\mathbb T}^n)\to L^{\frac{2(n+1)}{n-1}}({\mathbb T}^n)} \lesssim \la^{\e}, \, \, \, \forall \, \e>0.
 \end{equation}
 This represents a $1/q_c$--power improvement over the universal estimates \cite{sogge88} with $q_c=\frac{2(n+1)}{n-1}$.
 Similar to the case above in Corollary~\ref{cor1} if we use \eqref{m.1}
 with $B(\la)= \la^{n(\frac12-\frac1q)-\frac1{q_c}+\e}$, we obtain from \eqref{9} that
 if $P=\sqrt{-(\Delta_X+\Delta_{{\mathbb T}^n})}$ then
\begin{equation}\label{10}
\aligned
&\bigl\| \, \1_{[\la-\la^{-1}, \, \la+\la^{-1}]}(P) \, \bigr\|_{L^2(X\times {\mathbb T}^n)\to L^q(X\times {\mathbb T}^n)}
\lesssim \la^{d(\frac12-\frac1q)-1/2-1/q_c+\e}, \, \, \forall \, \e>0,
\\
&\text{if } \, \, q_c=\frac{2(n+1)}{n-1},\ \text{ and }\ d=\dx +n,
\endaligned
\end{equation}
and
\begin{equation}\label{11}
q\ge \max\bigl( \, \frac{2(n+1)}{n-1}, \, \frac{2(\dx+1)}{\dx-1}\, \bigr).
\end{equation}
It is conjectured that \eqref{9} should also be valid when $\tfrac{2(n+1)}{n-1}$ is replace by the larger exponent $\tfrac{2n}{n-2}$,
which would represents the optimal $\la^{-1/2+\e}$ improvement of the universal bounds in \cite{sogge88}.  If this result
held, then one would obtain the optimal bounds where in the exponent in \eqref{10} $-1/q_c$ is replaced by $-1/2$, which
would be optimal, as well as the range of exponents in \eqref{11}.

Also, using results of Hickman \cite{Hickman} and Germain and Myerson~\cite{Germain} one can also obtain improved spectral projection bounds when
$\el =\la^{-\sigma}$ with $\sigma\in (0,1)$.

Let us now present the proof  of Theorem~\ref{spthm} which in the case of $q=\infty$ strengthens the
bounds that are implicit in Iosevich and Wyman~\cite{IoWy}.

\begin{proof}[Proof of Theorem~\ref{spthm}]
Let $\{e_k^\mu\}_{\mu}$ be an orthonormal basis for spherical
harmonics of degree $k$ on $S^{d_1}$ and
$\{e_\ell^\nu\}_\nu$ be an orthonormal basis of
spherical harmonics of degree $\ell$ on $S^{d_2}$.
Then an orthonormal basis of eigenfunctions on
$S^{d_1}\times S^{d_2}$ is of the form
\begin{equation}\label{00.4}e_k \, e_\ell
\end{equation}
where $e_k=e_k^\nu$ for some $\nu$ and
$e_\ell =e_\ell^\mu$ for some $\mu$.
So,
$$(-\Delta_{d_1}+(\tfrac{d_1-1}2)^2)e_k
=(k+\tfrac{d_1-1}2)^2e_k$$
and
$$(-\Delta_{d_2}+(\tfrac{d_2-1}2)^2)e_\ell
=(\ell+\tfrac{d_1-1}2)^2e_\ell$$
so that for the Laplacian on $S^{d_1}\times S^{d_2}$,
$\Delta=\Delta_{d_1}+\Delta_{d_2}$, we have
$$(-\Delta +(\tfrac{d_1-1}2)^2+(\tfrac{d_2-1}2)^2)e_k
e_\ell
=((k+\tfrac{d_1-1}2)^2+ (\ell+\tfrac{d_1-1}2)^2)e_ke_\ell.$$

Thus if $P=\sqrt{(-\Delta +(\tfrac{d_1-1}2)^2+(\tfrac{d_2-1}2)^2)}$ its eigenvalues are
\begin{equation}\label{00.5}\la=\la_{k,\ell}=\sqrt{((k+\tfrac{d_1-1}2)^2+ (\ell+\tfrac{d_1-1}2)^2)}.\end{equation}
Thus, if $e_\lambda$ as in Theorem~\ref{spthm} is an eigenfunction of
$P$ with this eigenvalue we must have
\begin{equation}\label{00.6}
e_\la(x,y) = \sum_{\{(k,\ell): \, \la_{k,\ell}=\la\}}
\Bigl(\sum_{\mu,\nu} a^{\mu,\nu}_{k,\ell} \, e^{\mu}_k(x)e^{\nu}_\ell(y)\Bigr),
\quad (x,y)\in S^{d_1}\times S^{d_2}.
\end{equation}

Let us now prove \eqref{a.1}.  We note that if $(k,\ell)$
are fixed then, for every fixed $y\in S^{d_2}$, the function on $S^{d_1}$
$$x\to \sum_{\mu,\nu} a^{\mu,\nu}_{k,\ell} \, e^{\mu}_k(x)e^{\nu}_\ell(y)$$
is a spherical harmonic of degree $k$.  Thus by
\cite{sogge86} or \cite{sogge88}
$$\Bigl\| \sum_{\mu,\nu} a^{\mu,\nu}_{k,\ell} \, e^{\mu}_k(\, \cdot \, )e^{\nu}_\ell(y)
\Bigr\|_{L^q(S^{d_1})}
\le C\la^{\alpha(q,d_1)}
\Bigl\| \sum_{\mu,\nu} a^{\mu,\nu}_{k,\ell} \, e^{\mu}_k(\, \cdot \, )e^{\nu}_\ell(y)
\Bigr\|_{L^2(S^{d_1})}.$$
Next,
by Minkowski's inequality and another application of the universal bounds, we obtain from this
\begin{equation}
\label{00.7}
\aligned
\Bigl\| \sum_{\mu,\nu} a^{\mu,\nu}_{k,\ell} \, e^{\mu}_ke^{\nu}_\ell
\Bigr\|_{L^q(S^{d_1}\times S^{d_2})}
&\le C\la^{\alpha(q,d_1)}
\Bigl\| \sum_{\mu,\nu} a^{\mu,\nu}_{k,\ell} \, e^{\mu}_k(x)e^{\nu}_\ell(y)
\Bigr\|_{L^2_xL^q_y(S^{d_1}\times S^{d_2})}
\\
&\lesssim \la^{\alpha(q,d_1)+\alpha(q,d_2)}
\Bigl\| \sum_{\mu,\nu} a^{\mu,\nu}_{k,\ell} \, e^{\mu}_ke^{\nu}_\ell
\Bigr\|_{L^2(S^{d_1}\times S^{d_2})}.
\endaligned
\end{equation}
Since, by \eqref{a.4'}, the number of $\{(k,\ell): \, \la_{k,\ell}=\la\}$
is $O(\la^\e)$ we also have by the Cauchy-Schwarz
inequality that if $e_\la$ is as in \eqref{00.6}
\begin{equation}\label{00.8}
|e_\la(x,y)|
\lesssim \la^\e
\Bigl(\sum_{\{(k,\ell): \, \la_{k,\ell}=\la\}}\Bigl|
\Bigl(\sum_{\mu,\nu} a^{\mu,\nu}_{k,\ell} \, e^{\mu}_k(x)e^{\nu}_\ell(y)\Bigr)\Bigr|^2\, \Bigr)^{1/2}.
\end{equation}
Thus, by \eqref{00.7}--\eqref{00.8}
\begin{equation}
\label{00.9}
\aligned
&\|e_\la\|_{L^q(S^{d_1}\times S^{d_2})}
\\
&\lesssim  \la^{\alpha(q,d_1)+\alpha(q,d_2)+\e} \cdot
\Bigl(\sum_{\{(k,\ell): \, \la_{k,\ell}=\la\}}
\Bigl\| \sum_{\mu,\nu} a^{\mu,\nu}_{k,\ell} \, e^{\mu}_ke^{\nu}_\ell
\Bigr\|_{L^2(S^{d_1}\times S^{d_2})}^2\Bigr)^{1/2},
\endaligned
\end{equation}
which leads to \eqref{a.1} since, by orthogonality,
the last factor in \eqref{00.9} is $\|e_\la\|_{L^2}$.
\end{proof}

\newsection{Improved Weyl formulae}

To prove Theorem~\ref{weylthm},
we first observe that if as above
$\mu_i^2$ are the eigenvalues of $-\Delta_X$ then by \eqref{3.2} we have
\begin{equation}\label{3.4}
N(X\times Y, \la)=\sum_{\mui\le \la} \,
\Bigl[ (2\pi)^{-d_Y} \omega_{d_Y} (\text{Vol}\,Y) \bigl(\la^2-\mui^2\bigr)^{d_Y/2} +R_Y\bigl(\sqrt{\la^2-\mui^2}\bigr) \, \Bigr].
\end{equation}
We can estimate the last sum using \eqref{3.2} and \eqref{2'}:
\begin{align*}
R_\la& = \sum_{\mui\le \la} R_Y\bigl(\sqrt{\la^2-\mui^2}\bigr)
\\ &\lesssim \sum_{\mui\le \la} \e\bigl(\sqrt{\la^2-\mui^2}\bigr) \cdot \bigl(\la^2-\mui^2\bigr)^{\frac{d_Y-1}2}
\\
&\lesssim \la^{d_Y-1}  \sum_{\mui\le \la}  \e\bigl(\la\cdot \sqrt{1-\mui^2/\la^2} \bigr) \cdot \sqrt{1-\mui^2/\la^2}
\times \bigl(1-\mui^2/\la^2\bigr)^{\frac{d_Y-2}2}.
\end{align*}
Since $d_Y\ge2$, we have
$$\bigl(1-\mui^2/\la^2\bigr)^{\frac{d_Y-2}2}\le 1, \quad \text{if } \, \, \, \mui\le \la.$$
Thus, if we use \eqref{2'} with $\theta = \sqrt{1-\mui^2/\la^2}$ to estimate the terms with
$\sqrt{\la^2-\mui^2}\ge 1$, we get
$$ \e\bigl(\la\cdot \sqrt{1-\mui^2/\la^2} \bigr) \cdot \sqrt{1-\mui^2/\la^2}
\times \bigl(1-\mui^2/\la^2\bigr)^{\frac{d_Y-2}2}\le \e(\la) \quad
\text{if } \, \, \, \sqrt{\la^2-\mui^2}\ge 1.$$
Thus, by \eqref{3.1}.
\begin{equation*}
\sum_{\mui\le \la , \, \sqrt{\la^2-\mui^2}\ge 1} R_Y\bigl(\sqrt{\la^2-\mui^2}\bigr) \lesssim
\la^{d_Y-1}\e(\la) \sum_{\mui\le \la} 1
\lesssim \e(\la) \la^{d_Y-1} \cdot \la^{d_X} = \e(\la) \la^{d-1},
\end{equation*}
with, as before, $d=d_X+d_Y$.
We also need to estimate the terms where $ \sqrt{\la^2-\mui^2}\le 1$.  In this case we just use
that $R_Y(\theta) =O(1)$ if $\theta\le 1$ and so
$$\sum_{\mui\le \la , \, \sqrt{\la^2-\mui^2}\le 1} R_Y\bigl(\sqrt{\la^2-\mui^2}\bigr) \le
\sum_{\mui\le \la}1 \lesssim \la^{d_X} =\la^{d-1}\cdot \la^{1-d_Y},$$
and since $d_Y\ge 2$,
$$\la^{1-d_Y} \le \la^{-1}\le \e(\la).$$

By combining these two estimates we conclude that, if as above, $R_\la$ is the second sum in the right side
of \eqref{3.4} then
$$
R_\la =O(\e(\la) \, \la^{d-1}), \quad d=d_X+d_Y,
$$
as desired.

Based on this, we conclude that the improved Weyl formula \eqref{3.3} would be a consequence of the following
\begin{equation}
\label{3.5}
\aligned
&(2\pi)^{-d_Y} \omega_{d_Y} (\text{Vol}\,Y) \, \la^{d_Y} \cdot \sum_{\mui\le \la} \,
 \bigl(1-\mui^2/\la^2\bigr)^{d_Y/2}
 \\
&=(2\pi)^{-d}\omega_d (\text{Vol}\,Y\cdot\text{Vol}\,X) \la^d +O(\la^{d-2}).
\endaligned
\end{equation}

Note that $\sum_{\mui\le \la} \,
 \bigl(1-\mui^2/\la^2\bigr)^{d_Y/2}
$ is the trace of the kernel of $(1-(P_X)^2/\la^2)_+^{d_Y/2}$, i.e.,
\begin{equation}\label{3.6}
\sum_{\mui\le \la} \,
 \bigl(1-\mui^2/\la^2\bigr)^{d_Y/2}=
 \int_M \sum_{\mui\le \la} \,
 \bigl(1-\mui^2/\la^2\bigr)^{d_Y/2} e_{\mui}^X(x)\overline{e_{\mui}^X(x)} \, dV(x),
\end{equation}
 with $dV$ denoting the volume element on $X$.
For $\delta\ge0$
 \begin{equation}\label{3.7}
 S^\delta_\la(x,y) =\sum_{\mui\le \la} \bigl(1-\mui^2/\la^2\bigr)^{\delta} e_{\mui}^X(x)\overline{e_{\mui}^X(y)},
 \end{equation}
 denotes the kernel of the Bochner-Riesz operators $(1-(P_X)^2/\la^2)_+^\delta$ (see, e.g. \cite{soggefica}).
 Keeping \eqref{3.6} in mind, we claim that \eqref{3.5} (and hence Theorem~\ref{weylthm}) would be
 a consequence of the following pointwise estimates for these kernels restricted to the diagonal in $X\times X$.

 \begin{proposition}\label{traceprop}
 Let $S^\delta_\la$ be as in \eqref{3.7}.  Then if $\delta\ge1$ we have
 \begin{equation}\label{3.8}
 S^\delta_\la(x,x)=(2\pi)^{-d_X}  |S^{d_X-1}| \, \times  \frac12 B(\delta+1, d_X/2) \, \la^{d_X}
 +O(\la^{d_X-2}),
 \end{equation}
 with $|S^{d_X-1}|$ denoting the area of the unit sphere in ${\mathbb R}^{d_X}$ and
 $$B(s,t)=\int_0^1 (1-u)^{s-1} u^{t-1}\, du$$
 being the beta function.
 \end{proposition}

 To see that \eqref{3.6}--\eqref{3.8} imply \eqref{3.5} we recall the formulae
 $$\omega_n =\frac1n \, |S^{n-1}|  =\frac{\pi^{n/2}}{\Gamma(n/2 + 1)}
 ,$$
 and
 $$
 B(s,t)=\frac{\Gamma(s)\Gamma(t)}{\Gamma(s+t)}.
 $$
 Thus,
 \begin{align*}
 |S^{d_X-1}| \cdot  \frac12 B(d_Y/2+1,d_X/2)
&= \frac{d_X\pi^{d_X/2}}{\Gamma(d_X/2+1)}
\cdot \frac{\Gamma(d_Y/2+1)\Gamma(d_X/2)}{2\Gamma(d/2+1)}
\\
&= \frac{\pi^{d_X/2}\Gamma(d_Y/2+1)}{\Gamma(d/2+1)},
\end{align*}
and so
\begin{equation}
\label{3.9}
\aligned
\omega_{d_Y} \cdot|S^{d_X-1}| \cdot  \frac12 B(d_Y/2+2, d_X/2)
&=\frac{\pi^{d_Y/2}}{\Gamma(d_Y/2+1)} \cdot  \frac{\pi^{d_X/2}\Gamma(d_Y/2+1)}{\Gamma(d/2+1)},
\\
&=\frac{\pi^{d/2}}{\Gamma(d/2+1)} = \omega_d.
\endaligned
\end{equation}
Thus, since $d_Y/2\ge1$,  if \eqref{3.8} were valid, we would have
$$
\aligned
&(2\pi)^{-d_Y}w_{d_Y} (\text{Vol}\,Y) \la^{d_Y} \times \int_XS^{d_Y/2}_\la(x,x) \, dV(x)
\\
&=(2\pi)^{-d}\la^{d_Y}\omega_d (\text{Vol}\, Y\cdot\text{Vol}\, X)
\cdot \la^{d_X} + O(\la^{d_Y}\cdot \la^{d_X-2})
\\
&=(2\pi)^{-d} \omega_d\, \text{Vol}(X\times Y)  \la^d +O(\la^{d-2}),
\endaligned
$$
Since, by \eqref{3.6} and \eqref{3.7}, this yields \eqref{3.5}
, we conclude that the proof of Theorem~\ref{weylthm} would be complete
if we could establish Proposition~\ref{traceprop}.

\begin{proof}[Proof of Proposition~\ref{traceprop}]
The proof of kernel estimates for Bochner-Riesz estimates are well known.  See, e.g., \cite{HormanderIII}, \cite{sogge87}, \cite{soggefica}
and \cite{SoggeHangzhou}.    We shall adapt the argument in the latter reference, which is based on arguments that
exploit the Hadamard parametrix and go back
to Avakumovic~\cite{Avakumovic} and Levitan~\cite{Levitan}.   These dealt with the analog of \eqref{3.8} where
$\delta=0$ and then the  error bounds in
\eqref{3.8} must be replaced by $O(\la^{d_X-1})$.

To proceed, let $m_\delta(\tau)$, $\tau\in \R$, denote
the even function
$$\tau \to m_\delta(\tau) = (1-\tau^2)_+^\delta.$$
Then, if $\Hat m_\delta$ denotes its Fourier transform,
we have by Fourier's inversion theorem
\begin{equation*}
S^\delta_\la f(x) =m_\delta(P_X)f(x)
=\frac1\pi \int_0^\infty \la \Hat m_\delta(\la t)
\bigl(\cos tP_X\bigr)f(x) \, dt.
\end{equation*}
Thus,
\begin{equation}
\label{3.10}
\aligned
S^\delta_\la(x,y)
&=\frac1\pi \int_0^\infty \la \Hat m_\delta(\la t)
\bigl(\cos tP_X\bigr)(x,y) \, dt
\\
&=\frac1\pi \sum_i
\int_0^\infty \la \Hat m_\delta(\la t) \,
\cos t\mu_i \, e^X_{\mu_i}(x)\,
\overline{e^X_{\mu_i}(y)} \,  dt.
\endaligned
\end{equation}

To be able to exploit this, we require a couple
of facts about $m_\delta$.  First, we can write its
Fourier transform as follows
\begin{equation}
\label{3.11}
\aligned
&\Hat m_\delta(t) = a_0^\delta(t)+ a_+^\delta(t)e^{it}
+a_-^\delta(t) e^{-it}, \quad \text{where}
\\
&|\partial_t^j a(t)|
\lesssim O((1+|t|)^{-1-\delta-j}) \, \, \forall \, \,
j=0,1,2,\dots, \quad
\text{if } \, a=a_0, a_+, a_-.
\endaligned
\end{equation}
Also,
\begin{equation}\label{3.12}
\int_0^\infty m_\delta(r) \, r^{d_X-1} \, dr
=\frac12 B(\delta+1, d_X/2).
\end{equation}
Let us postpone the simple proofs of these two facts for a moment and see how they can be used, along with the
Hadamard parametrix, to prove Proposition~\ref{traceprop}.

Let us first fix an even function $\rho(t)\in C^\infty_0(\R)$ satisfying the following
\begin{equation}\label{3.13}
\text{supp }\rho \subset (-c/2,c/2) \quad
\text{and } \, \rho\equiv 1 \, \, \text{on } \, \,
[-c/4,c/4],
\end{equation}
where we assume that
$$c= \min\{1, \, \text{Inj }X/2\},$$
with $\text{Inj $X$}$ denoting the injectivity
radius of $(X,g_X)$.
It follows from \eqref{3.11} that
\begin{equation}
\label{3.14}
\aligned
r_\la^\delta(\mu)
&=\frac1\pi \int_0^\infty (1-\rho(t)) \,
\la \Hat m_\delta(\la t) \cos t\mu \, dt
\\
&=O\bigl(\la^{-\delta}(1+|\la -\mu|)^{-N}\bigr), \, \,
\forall N, \, \, \quad \text{if }\, \, \mu\ge0.
\endaligned
\end{equation}
Thus, if we modify the kernels in \eqref{3.10} as follows
\begin{equation}
\label{3.15}
\aligned
\widetilde S^\delta_\la(x,y)
&=\frac1\pi \int_0^\infty \rho(t) \la \Hat m_\delta(\la t)
\bigl(\cos tP_X\bigr)(x,y) \, dt
\\
&=\frac1\pi \sum_i
\int_0^\infty \rho(t) \la \Hat m_\delta(\la t) \,
\cos t\mu_i \, e^X_{\mu_i}(x)\,
\overline{e^X_{\mu_i}(y)} \,  dt,
\endaligned
\end{equation}
and let
$$R^\delta_\la(x,y)=S^\delta_\la(x,y)-\widetilde S^\delta_\la(x,y).$$
It follows that
\begin{equation}\label{3.16}
R^\delta_\la(x,y)=\sum_i r_\la(\mu_i)e^X_{\mu_i}(x)
\overline{e^X_{\mu_i}(y)},
\end{equation}
satisfies
\begin{equation}\label{3.17}
|R^\delta_\la(x,y)|\lesssim \la^{-\delta}
\sum_i \bigl(1+|\la-\mu_i|\bigr)^{-N} |e^X_{\mu_i}(x)| \,
|e^X_{\mu_i}(y)|, \quad N=1,2,\dots .
\end{equation}
As is well known, the pointwise Weyl formula of
Avakumovic~\cite{Avakumovic} and Levitan~\cite{Levitan}
(see also \cite{HSpec}, \cite{SoggeHangzhou}) yields
the uniform bounds
$$
\sum_{\mu_i\in [\tau, \tau+1)} |e^X_{\mu_i}(x)|^2
\lesssim (1+\tau)^{d_X-1}, \quad \tau \ge0,
$$
which in turn give us
\begin{equation}\label{3.18}R^\delta_\la(x,y)
=O(\la^{d_X-1-\delta})=O(\la^{d_X-2}),
\end{equation}
since we are assuming in Proposition~\ref{traceprop}
that $\delta\ge1$.

Consequently, it suffices to show
that $\widetilde S^\delta_\la(x,x)$ equals the first
term in the right side of \eqref{3.8} up to error terms
which are $O(\la^{d_X-2})$.

To do this, if $d_X\ge 2$, we recall that the Hadamard parametrix
implies that
for $|t|$ smaller than half the injectivity radius
of $X$ we can write
\begin{multline}\label{3.19}
 \bigl(\cos tP_X\bigr)(x,x)
=(2\pi)^{-d_X} \int_{{\mathbb R}^{d_X}} \cos t|\xi| \,
d\xi + \alpha_0(x)\int_{{\mathbb R}^{2}} t \frac{\sin t|\xi|}{|\xi|} d\xi
\\
+ \int_{{\mathbb R}^{d_X}} \cos t|\xi| \, \alpha_1(t,x,\xi)\, d\xi
+ \int_{{\mathbb R}^{d_X}} \sin t|\xi| \, \alpha_2(t,x,\xi)\, d\xi +O(1),
\end{multline}
where $\alpha_0$ is a smooth function, and the $\alpha_j$ are symbols of order $-3$, so, in particular
$$|\partial^\gamma_\xi \alpha_j|\lesssim (1+|\xi|)^{-3-|\gamma|}.$$

We shall first deal with the second term on the right side. If we take $x=y$ in \eqref{3.15} and
replace $\bigl(\cos tP_X\bigr)(x,x)$ by the second term in the right side of \eqref{3.19}, our goal is to show that

\begin{equation}\label{3.20}
\int_0^\infty
\int_{{\mathbb R}^{d_X}} t \frac{\sin t|\xi|}{|\xi|}\,  \la \Hat m_\delta(\la t) \rho (t)d\xi dt \lesssim \la^{d_X-2}.
\end{equation}

To handle the part of the integral where $|\xi|\ge 2\la$, note that since $$\Hat m_\delta(\la t)=\frac1\pi\int_{-1}^{1} e^{-it\la\tau}m_\delta(\tau)d\tau$$ is an even function in $t$, it suffices to show that
\begin{equation}\label{3.21}
\int_{|\xi|\ge 2\la} \int_{-\infty}^\infty\int_{-1}^{1}
t \frac{\sin t|\xi|}{|\xi|}\,  \la e^{-it\la\tau}m_\delta(\tau) \rho (t)  \,dt d\tau d\xi  \lesssim \la^{d_X-2}.
\end{equation}
However, by integrating by parts in $t$,  if $|\xi|\ge 2\la$, it is easy to see that
\begin{equation}\label{3.22}
 \int_{-\infty}^\infty\int_{-1}^{1}
t \frac{\sin t|\xi|}{|\xi|}\,  \la e^{-it\la\tau}m_\delta(\tau) \rho (t)  \,dt d\tau  \lesssim O(1+|\xi|)^{-N}, \,\,\,\forall\,N,
\end{equation}
which clearly implies \eqref{3.21}.

For the part of integral where $|\xi|\le2\la$,  let us fix $\eta \in C_0^\infty$ satisfying
\begin{equation}\nonumber
\text{supp }\eta \subset (1/2,\infty) \quad
\text{and } \, \eta \equiv 1 \, \, \text{on } \, \,
[1, +\infty),
\end{equation}
and write
\begin{equation}
\begin{aligned}
\int_0^\infty&\int_{|\xi|\le  2\la} t \frac{\sin t|\xi|}{|\xi|}\,  \la \Hat m_\delta(\la t) \rho (t)d\xi dt \\
&=\int_0^\infty\int_{|\xi|\le  2\la}\big(1-\eta(t|\xi|)\big) t \frac{\sin t|\xi|}{|\xi|}\,  \la \Hat m_\delta(\la t) \rho (t)d\xi dt \\
&\qquad+\int_0^\infty\int_{|\xi|\le  2\la}\eta(t|\xi|) t \frac{\sin t|\xi|}{|\xi|}\,  \la \Hat m_\delta(\la t) \rho (t)d\xi dt  \\
&=I+II.
\end{aligned}
\end{equation}

For the first term on the right, note that $|\xi|\le \min\{t^{-1}, 2\la\}$, and so by \eqref{3.11}
\begin{equation}
\begin{aligned}
I&\lesssim \int_{t\ge (2\la)^{-1}}\int_{|\xi|\le t^{-1}}  \frac{1}{(t\la)^{\delta}|\xi|}\,d\xi dt \\
&\qquad+ \int_{t\le (2\la)^{-1}}\int_{|\xi|\le 2\la} \frac{1}{|\xi|}\, d\xi dt\\
& \lesssim \int_{t\ge (2\la)^{-1}}   t^{1-d_X-\delta}\,\la^{-\delta}  dt + \int_{t\le (2\la)^{-1}}   \,\la^{d_X-1}  dt \\
& \lesssim \la^{d_X-2}.
\end{aligned}
\end{equation}

To bound $II$, by integrating by parts in $t$, we rewrite it as
\begin{equation}
\begin{aligned}
\int_{|\xi|\le  2\la}&\int_0^\infty\eta(t|\xi|)  \frac{\sin t|\xi|}{|\xi|}\, t \la \Hat m_\delta(\la t) \rho (t)d\xi dt  \\
&=\int_{|\xi|\le  2\la}\int_0^\infty\eta(t|\xi|) \frac{\cos t|\xi|}{|\xi|^2}\,  t\la \Hat m_\delta(\la t) \rho' (t)d\xi dt  \\
&\qquad +\int_{|\xi|\le  2\la}\int_0^\infty \eta'(t|\xi|) \frac{\cos t|\xi|}{|\xi|}\,  t\la \Hat m_\delta(\la t) \rho (t)d\xi dt \\
&\qquad +\int_{|\xi|\le  2\la}\int_0^\infty \eta(t|\xi|) \frac{\cos t|\xi|}{|\xi|^2}\,  \big(t\la \Hat m_\delta(\la t)\big)' \rho (t)d\xi dt \\
&=I+II+III.
\end{aligned}
\end{equation}

For the first term, since $\rho'(t)$ is supported where $t\approx 1$, by \eqref{3.11}, we have
\begin{align*}
I
&\lesssim \int_{|\xi|\le 2\la}  \frac{1}{|\xi|^2}\la^{-\delta} d\xi  \\
 & \lesssim \la^{d_X-2}, \,\,\,\text{if}\,\,\,\delta>0.
\end{align*}
For the second term, since $\eta'(t|\xi|)$ is supported where $t\approx |\xi|$, by \eqref{3.11}, we have
\begin{align*}
II&\lesssim \int_{|\xi|\le 2\la } \int_{t\approx |\xi|^{-1}} \frac{1}{(t\la)^{\delta}|\xi|}\,dt d\xi
\\
&\lesssim \int_{|\xi|\le 2\la}  \frac{1}{|\xi|^{2-\delta}}\la^{-\delta} d\xi  \\
 & \lesssim \la^{d_X-2}.
\end{align*}
For the third term,  we use the fact that by \eqref{3.11},  $\big(t\la  \Hat m_\delta(\la t)\big)'\lesssim \la^{-\delta}  t^{-1-\delta}$, which implies
\begin{align*}
III&\lesssim\int_{|\xi|\le 2\la} \int  \eta(t|\xi|)  \frac{1}{|\xi|^2}\, \la^{-\delta}  t^{-1-\delta}\rho(t)dt d\xi
\\
&\lesssim \int_{|\xi|\le 2\la}   \frac{1}{|\xi|^{2-\delta}}\la^{-\delta} d\xi  \\
 & \lesssim \la^{d_X-2}.
\end{align*}
Thus the proof of \eqref{3.20} is complete.

On the other hand,
if we take $x=y$ in \eqref{3.15} and
replace $\bigl(\cos tP_X\bigr)(x,x)$ by the third or fourth
terms in the right side of \eqref{3.19}, one can use
\eqref{3.11} to see that the resulting expression must
be bounded by
$$
\aligned
\int_{\{\xi\in{\mathbb R}^{d_X}: \, |\xi|\le 2\la \}}
&(1+|\xi|)^{-3}\, d\xi
+ \int_{\{\xi\in{\mathbb R}^{d_X}: \, |\xi|\ge 2\la \}}
(1+|\la-|\xi|\, |)^{-N} \, d\xi
\\ &=
\begin{cases}
O(\la^{d_X-3}), \quad d_X> 3, \,\,
\\
O(\log \la), \quad d_X=3, \\
O(1), \quad d_X=2,
\end{cases}
\endaligned
$$
which is better than desired. Clearly, if we also do this for the last term in the right side of \eqref{3.19}, the resulting term will
be $O(1)$.

Based on this, we would have the bounds in Proposition~\ref{traceprop} for $d_X\ge 2$ if we could show that
\begin{equation}
\label{3.26}
\aligned
&(2\pi)^{-d_X} \int_0^\infty\int_{{\mathbb R}^{d_X}}
\frac1\pi
\rho(t) \la \Hat m_\delta(t) \,  \cos t|\xi| \,
d\xi dt
\\
&= (2\pi)^{-d_X}  |S^{d_X-1}| \, \times  \frac12 B(\delta+1, d_X/2) \, \la^{d_X}
 +O(\la^{d_X-2}).
\endaligned
\end{equation}
If we repeat the argument that lead to \eqref{3.17},
we find if we replace $\rho(t)$ here by one then the
difference between this expression and the left
side of \eqref{3.26} is bounded by
$$\la^{-\delta}\int_{{\mathbb R}^{d_X}}
(1+|\la -|\xi|\, |)^{-N} \, d\xi$$
for any $N$ and hence $O(\la^{d_X-1-\delta})=O(\la^{d_X-2})$.

 Thus, by Fourier's inversion formula, up to these
errors, the expression in the left side of \eqref{3.26} is
\begin{equation}\label{3.27}
(2\pi)^{-d_X} \int_{{\mathbb R}^{d_X}}
m_\delta(|\xi|/\la) \, d\xi
=(2\pi)^{-d_X} |S^{d_X-1}| \,
\Bigl(\, \int_0^1 m_\delta(r) \, r^{d_X-1} \, dr
\, \Bigr) \cdot\la^{d_X}.
\end{equation}
Since by \eqref{3.12} the integral in the right side is equal to
$\frac12 B(\delta+1, d_X/2)$, we obtain \eqref{3.26}.

For the remaining case $d_X=1$, we shall use the fact that for $|t|$ smaller than a fixed constant $c$, by choosing coordinates such that
the metric equals $dx^2$, we have
\begin{equation}\label{3.28}
 \bigl(\cos tP_X\bigr)(x,y)
=(2\pi)^{-1} \int_{-\infty}^\infty \cos t\tau e^{i\tau(x-y)} \,d\tau\,.
\end{equation}
Moreover, one can simply repeat the argument in \eqref{3.20}-\eqref{3.21} to see that,
\begin{equation}\label{3.29}
(2\pi)^{-1} \int_0^\infty\int_{-\infty}^\infty
\frac1\pi
\rho(t) \la \Hat m_\delta(t) \,  \cos t\tau \,
d\tau dt
= (2\pi)^{-1}  B(\delta+1, 1/2) \, \la
 +O(\la^{-1}).
\end{equation}
By \eqref{3.28}, \eqref{3.29} and the arguments in \eqref{3.13}-\eqref{3.18}, we obtain \eqref{3.8} when $d_X=1$.

To finish, we still need to prove the facts \eqref{3.11} and \eqref{3.12} about $m_\delta$ that we used above.
The latter just follows from the standard formula for the beta function stated above and a change of variables.
To prove the former, \eqref{3.11} we note that
if $\rho$ is as in \eqref{3.13} then
$\Hat m_\delta(t)$ can
be written as
$$\int \rho(1-\tau) \, (1-\tau)_+^\delta (1+\tau)^\delta
\, e^{-it\tau}\, d\tau
+
\int \rho(1+\tau) \, (1+\tau)_+^\delta (1-\tau)^\delta
\, e^{-it\tau}\, d\tau +a_0^\delta(t),$$
where $a_0^\delta \in {\mathcal S}(\R)$ and hence
satisfies the bounds in \eqref{3.11}.  A simple argument
shows that the first two terms in the right can
be written as $a^\delta_+(t)e^{it}$ and
$a^\delta_-(t)e^{-it}$, respectively, with $a^\delta_\pm$ as in \eqref{3.11}, which finishes the proof.
\end{proof}

%

\newsection{Further results and remarks}  

In our main results, Theorem~\ref{mainthm}, Theorem~\ref{weylthm} and Theorem~\ref{spthm}, we focused on products of
length two, as was the case of some of the earlier results, e.g., \cite{CZWeyl} and \cite{CGBeams}.  On the other hand, Iosevich
and Wyman~\cite{IoWy} obtained further improved Weyl error bounds for products of spheres $S^{d_1}\times S^{d_2}\times \cdots
\times S^{d_n}$ as the length $n=2,3,\dots$ increased, and their $O(\la^{d-1-\delta_n})$ bounds, $d=d_1+\cdots+d_n$,
have $\delta_n\to 1$ as $n\to \infty$.  As we noted earlier, such bounds are impossible for $\delta_n>1$.  Iosevich and Wyman
conjectured that for such products of length $n\ge5$ one should be able to take $\delta_n=1-\e$ for all $\e>0$ or even
$\delta_n=1$, which would agree with the classical error term bounds for the $n$-torus (i.e., $d_1=\cdots =d_n=1$ and $n\ge5$).  See
e.g. Walfisz~\cite{latticebook}.

Let us now show that the proof of Theorem~\ref{spthm} yields optimal $L^q$ estimates for such products with $q$ large.  The particular
case where $q=\infty$ can be thought of as a weaker version of the conjecture of Iosevich and Wyman~\cite{IoWy} in the sense that
it would follow from the somewhat stronger pointwise Weyl remainder variant of their conjecture.

The improved variant of Theorem~\ref{spthm} that follows from its proof and the aforementioned optimal bounds
for ${\mathbb T}^n$ for $n=5$, is the following.

\begin{theorem}\label{5thm}  Let $Y=S^{d_1}\times \cdots \times S^{d_5}$ be a product of $5$ round spheres and let
$-\Delta_Y=-(\Delta_{S^{d_1}}+\cdots + \Delta_{S^{d_5}})$ and $P_Y=\sqrt{-\Delta_Y}$.  We then have for $\la\ge1$
\begin{equation}
\label{5.1}
\aligned
&\bigl\| \, \lala(P_Y) \, \bigr\|_{L^2(Y)\to L^q(Y)}=O(\la^{d(\frac12-\frac1q)-1}),
\\
&\text{if } \, d=d_1+\cdots+d_5 \, \, \text{and } \, \, q\ge \max\bigl\{ \tfrac{2(d_j+1)}{d_j-1}: \, 1\le j\le 5\bigr\}.
\endaligned
\end{equation}
Additionally, if $X=M^n$ is an $n$-dimensional, $n\ge1$, compact Riemannian manifold and $P=\sqrt{-(\Delta_Y+\Delta_X)}$ then
for $\la\ge1$
\begin{equation}
\label{5.2}
\aligned
&\bigl\| \, \lala(P) \, \bigr\|_{L^2(X\times Y)\to L^q(X\times Y)}=O(\la^{(d+n)(\frac12-\frac1q)-1}),
\\
&\text{if } \,  q\ge \max \{ \tfrac{2(d_1+1)}{d_1-1},
\tfrac{2(d_2+1)}{d_2-1}, \dots, \tfrac{2(d_5+1)}{d_5-1}, \tfrac{2(n+1)}{n-1} \bigr\}.
\endaligned
\end{equation}
\end{theorem}

Both estimates represents a $\la^{-1/2}$ improvement over the universal bounds in \cite{sogge88}, and, as mentioned before,
this is optimal.

To prove the Theorem, we first note that the second estimate, \eqref{5.2}, is a simple consequence
of the first one, \eqref{5.1}, and Theorem~\ref{mainthm} after noting that $\alpha(q,n)=n(\tfrac12-\tfrac1q)-\tfrac12$ for
$q$ as in \eqref{5.2}.

Let us now see how we can use the proof of Theorem~\ref{spthm} and the classical improved lattice point
counting bounds in dimension $5$ to obtain \eqref{5.1}.

Just as we did before for products of length $2$, we first note that if $\{e^\nu_{j,k}\}_\nu$ is an orthonormal basis
of spherical harmonics of degree $k$ on $S^{d_j}$ then an orthonormal basis of eigenfunctions on
$Y=S^{d_1}\times \cdots \times S^{d_5}$ is of the form
$$e_{1,k_1} \cdot e_{2,k_2}   \cdot  e_{3,k_3}  \cdot e_{4,k_4} \cdot e_{5,k_5},$$
where $e_{j,k_j}=e^{\nu_j}_{j,k_j}$ for some $\nu_j$, with $j=1,2,3,4,5$.  Consequently,
$$
\aligned
&\Bigl(-\Delta_Y+\bigl(\tfrac{d_1-1}2\bigr)^2+\cdots+\bigl(\tfrac{d_5-1}2\bigr)^2\Bigr)
\, \bigl[ e_{1,k_1}\cdots e_{5,k_5}\bigr]
\\
&= \Bigl( \, \bigl(k_1+\tfrac{d_1-1}2\bigr)^2+\cdots + \bigl(k_5+\tfrac{d_5-1}2\bigr)^2 \, \Bigr) \,  e_{1,k_1}\cdots e_{5,k_5},
\endaligned
$$
and so, analogous  to \eqref{00.5}, the eigenvalues of $P_Y$ are
\begin{equation}\label{5.3}
\la=\la_{k_1,\dots,k_5}=
\sqrt{\bigl(k_1+\tfrac{d_1-1}2\bigr)^2+\cdots + \bigl(k_5+\tfrac{d_5-1}2\bigr)^2},
\, \, \, k_j=0,1,2,\dots, \, \, 1\le j\le 5.
\end{equation}

Also, analogous to before, an eigenfunction with this eigenvalue must be of the form
\begin{equation}\label{5.4}
e_\la(x_1,\dots,x_5)=
\sum_{\{(k_1,\dots,k_5): \, \la_{k_1,\dots,k_5}=\la\}}
\Bigl( \sum_{\nu^1_{j_1}, \dots, \nu^5_{j_5}} a^{\nu^1_{j_1}, \dots, \nu^5_{j_5}}_{k_1,\dots,k_5}
e^{\nu^1_{j_1}}_{k_1}(x_1)\cdots e^{\nu^5_{j_5}}_{k_5}(x_5)\Bigr).
\end{equation}
Here $\{ e^{\nu^\ell_{j_\ell}}_{k_\ell}\}_{j_\ell}$ is the orthonormal basis of spherical harmonics of degree $k_\ell$ on
$S^{d_\ell}$, $\ell=1,\dots,5$.

Next, we note that for $q$ as in \eqref{5.1}, we have that if, as in \eqref{4}, $\alpha(q,d_j)$ denotes
the $\la$-power in the universal $L^q$-estimates then
\begin{equation}\label{5.5}
\alpha(q,d_j)=d_j(\tfrac12-\tfrac1q)-\tfrac12,
\end{equation}
if $q$ is as in \eqref{5.1}.  Thus, if we inductively use the universal bounds from \cite{sogge88} (or the earlier
bounds for spherical harmonics \cite{sogge86}), we find that if $k_1,\dots,k_5$ are fixed and
$\la_{k_1,\dots,k_5}=\la$
\begin{equation}
\label{5.6}
\aligned
&\Bigl\| \sum_{\nu^1_{j_1}, \dots, \nu^5_{j_5}} a^{\nu^1_{j_1}, \dots, \nu^5_{j_5}}_{k_1,\dots,k_5}
e^{\nu^1_{j_1}}_{k_1} \cdots e^{\nu^5_{j_5}}_{k_5} \, \Bigr\|_{L^q(Y)}
\\
&\le C\Bigl(\prod_{j=1}^5 \la^{d_j(\tfrac12-\tfrac1q)-\tfrac12}\Bigr) \cdot \Bigl(\,
 \sum_{\nu^1_{j_1}, \dots, \nu^5_{j_5}} \bigl|a^{\nu^1_{j_1}, \dots, \nu^5_{j_5}}_{k_1,\dots,k_5}\bigr|^2 \, \Bigr)^{1/2}
 \\
 &=C\la^{d(\tfrac12-\tfrac1q)-\tfrac52}
 \Bigl(
  \sum_{\nu^1_{j_1}, \dots, \nu^5_{j_5}} \bigl|a^{\nu^1_{j_1}, \dots, \nu^5_{j_5}}_{k_1,\dots,k_5}\bigr|^2 \, \Bigr)^{1/2}.
\endaligned
\end{equation}

To use this we recall that when $n\ge5$ we have the following improvement of \eqref{a.7}
\begin{equation}\label{5.7}
\#\{j\in {\mathbb Z}^n: \, |j|=\la \}=O(\la^{n-2}), \quad \text{if } \, \, n\ge 5.
\end{equation}
Indeed, this is a consequence of the stronger result for the problem of counting
the number of integer lattice points inside $\la$-balls centered at the
origin (e.g. \cite[p. 45]{latticebook}).

If we use \eqref{5.7} along with Cauchy-Schwarz inequality we deduce that \eqref{5.6} implies
that if $e_\la$ is as in \eqref{5.4}
\begin{align}\label{5.1'}\tag{4.1$'$}
\|e_\la\|_{L^q(Y)}&\le C\la^{d(\tfrac12-\tfrac1q)-\tfrac52} \, \sqrt{\la^3} \,
 \Bigl(
  \sum_{\nu^1_{j_1}, \dots, \nu^5_{j_5}} \bigl|a^{\nu^1_{j_1}, \dots, \nu^5_{j_5}}_{k_1,\dots,k_5}\bigr|^2 \, \Bigr)^{1/2}
  \\
  &=C\la^{d(\tfrac12-\tfrac1q)-1} \|e_\la\|_{L^2(Y)}. \notag
  \end{align}

  As before, this estimate for eigenfunctions implies the spectral projection bounds due to the fact that successive distinct
  eigenvalues of $P_Y$ which are comparable to $\la$ have gaps that are bounded below by $c_0\la^{-1}$ for some
  uniform $c_0>0$.  This completes the proof of Theorem~\ref{5thm}.

  \medskip

  \noindent {\bf Remark.}  It would be interesting to investigate other situations involving product manifolds where one is able
  to obtain  $L^q$ estimates that improve ones that follow from Theorem~\ref{mainthm}.  For instance
  Canzani and Galkowski~\cite{CGBeams} showed that if $Y$ is a product manifold then one has
  $\sqrt{\log\la}$ improvements over the universal bounds for large $q$ (i.e., $\e(\la)=(\log \la)^{-1}$ in Theorem~\ref{mainthm}).
  In this case, $X\times Y$ in Theorem~\ref{mainthm} would be a product of three manifolds, yet our
  results do not give further improvements over the results coming from \cite{CGBeams}.  Similarly, if {\em both} $X$ and $Y$
  have improved eigenfunction bounds
  are there situations where
    $X\times Y$ can inherit both improvements, as opposed to the better of the two
  improvements for $X$ and $Y$ as guaranteed by Theorem~\ref{mainthm}?  Our proof does not seem to yield such
  a result.  Moreover, in many cases one cannot obtain $\sqrt{\e_X(\la)} \cdot \sqrt{\e_Y(\la)}$ improvements if
  $X$ and $Y$, respectively, have $\sqrt{\e_X(\la)}$ and  $\sqrt{\e_Y(\la)}$ improvements.  This is the case, for instance,
  when they both involve power improvements $\sqrt{\e_X(\la)}=\la^{-\delta_X}$ and $\sqrt{\e_Y(\la)}=\la^{-\delta_Y}$
  with $\delta_X+\delta_Y>1/2$, for reasons mentioned before.


\bibliography{refs}
\bibliographystyle{abbrv}

%

\end{document}